\documentclass[12pt]{article}
\usepackage{amsmath, amscd, amsfonts, amssymb, amsthm, tikz, times, setspace, srcltx, verbatim}
\newtheorem{theorem}{Theorem}[section]
\newtheorem{lemma}[theorem]{Lemma}
{\bf}{\rm}
    \numberwithin{remarks}{section}

    \newtheorem{proposition}[theorem]{Proposition}
\newtheorem{definition}[theorem]{Definition}
\numberwithin{theorem}{section} 

\newcommand{\eqnsection}{
    \renewcommand{\theequation}{\thesection.\arabic{equation}}
    \makeatletter
    \csname @addtoreset\endcsname{equation}{section}
    \makeatother}

\newcommand{\dd}{\delta}

\newcommand{\lar}{\longrightarrow}
\newcommand{\eps}{\varepsilon}

\newcommand{\aaa}{\alpha}
\newcommand{\reals}{\mathbb{R}}

\def \be{\begin{equation}}
\def \ee{\end{equation}}
\def \bt{\begin{theorem}}
\def \et{\end{theorem}}
\def \bl{\begin{lemma}}
\def \el{\end{lemma}}
\def \bea{\begin{eqnarray}}
\def \eea{\end{eqnarray}}
\def \bas{\begin{eqnarray*}}
\def \eas{\end{eqnarray*}}
\def \ddd{\diamond}

\def \lll{\label}
\newcommand {\rrr}[1]{(\ref{#1})}

\def \aa{\alpha}
\def \bb{\beta}
\def \ga{\gamma}

\def \la{\lambda}
\def \lan{\langle}
\def \ran{\rangle}

\def \ppp{\rho}

\def \ff{\infty}

\def \R{\mathbb{R}}

\def \N{\mathbb{N}}

\def \H{{\bf H}}

\def \a{\alpha_t'(y)}

\def \BB{{\cal B}}

\def \DD{{\cal D}}

\def \HH{{\cal H}}

\def \PP{{\cal P}}

\def \SS{{\cal S}}

\def \E{\mathbf{E}}

\def \({\left(}
\def \){\right)}
\def \lk{\left[}
\def \rk{\right]}

\def \nn{\nonumber}

\def \vski{\vspace{12pt}}
\def \sgn{\mathop{\mathrm{sgn}}}

\def \bc{\begin{center} }
\def \ec{\end{center} }
\def \bs{\begin{slide} }
\def \es{\end{slide} }

\def\square{{\vcenter{\vbox{\hrule height.3pt
         \hbox{\vrule width.3pt height5pt \kern5pt
            \vrule width.3pt}
         \hrule height.3pt}}}}
\eqnsection
\begin{document}
\title{On the Tanaka formula for the derivative of self-intersection local time of fractional Brownian motion}
\author{Paul Jung \quad Greg Markowsky}
\maketitle

\abstract{The derivative of self-intersection local
time (DSLT) for Brownian motion was introduced by Rosen \cite{rosen2005} and subsequently used by others to study the $L^2$ and $L^3$ moduli
of continuity of Brownian local time. A version of the DSLT for fractional Brownian motion (fBm) was introduced in \cite{yan2008}; however, the definition given there presents difficulties, since it is motivated by an incorrect application of It\^o's formula.
To rectify this, we introduce a modified DSLT for fBm and prove existence using an explicit Wiener chaos
expansion. We will then argue that our modification is the natural version of the DSLT by rigorously proving the corresponding Tanaka formula. This formula corrects a formal identity given in both \cite{rosen2005} and \cite{yan2008}.
In the course of this endeavor we prove a Fubini theorem for integrals with respect to fBm. The Fubini theorem may be of independent interest,
as it generalizes (to Hida distributions) similar results previously seen in the literature. As a further byproduct of our investigation, we also provide a correction to an important technical second-moment bound for fBm given in \cite{hu2001}.}

\tableofcontents
\section{Introduction}

Let $B_t$ denote a Brownian motion on $\R$ with $B_0=0$ and let
$L(t,x)=L^x_t$ be its local time at $x$ up to time $t$. In
connection with stochastic area integrals with respect to local time
\cite{walsh1983} and the Brownian excursion filtration, Rogers and
Walsh \cite{rogers1991local} studied the space integral of local
time,

\be A(t,x) := \int_0^t 1_{[0,\infty)}(x-B_r) \,dr \ \ ;\
\frac{\partial A}{\partial x} = L(t,x). \ee
In a companion work on
the Brownian local time sheet \cite{rogers1991instrinsic}, the
occupation density of $A(t,B_t)$ was investigated. It was also shown in \cite{rogers1990} that the
process $A(t,B_t) - \int_0^t L(s,B_s)
dB_s$ has finite, non-zero
4/3-variation. An alternate proof of this fact using fractional martingales was recently given in \cite{hu2012frac}.

In \cite{rosen2005}, Rosen developed a new approach to the study of $A(t,B_t)$.
The motivation and guiding intuition was as follows. If one lets $h(x) := 1_{[0,\ff)}(x)$,
then formally
\be  \frac{d}{dx} h(x) = \dd(x)\quad\text{and}\quad\frac{d^2}{dx^2} h(x) = \dd'(x),
\ee
where $\dd$ is the Dirac delta distribution. Holding $r$ fixed and applying It\^o's formula with respect to the Brownian motion $B_s-B_r$ gives
\be \lll{cleopatra}
1_{[0,\ff)}(B_t - B_r) - 1_{[0,\ff)}(0) = \int_r^t \dd(B_s - B_r)\, dB_s + \frac{1}{2} \int_r^t \dd'(B_s - B_r)\, ds.
\ee
Integrating with respect to $r$ from $0$ to $t$ and interchanging the order of integration leads to

\be \nn
\label{eq:Ito applicationcleo}
\int_0^t 1_{[0,\infty)}(B_t-B_r) \,dr - t = \int_0^t \int_{0}^{s}
\delta(B_s-B_r) \, dr\, dB_s + \frac{1}{2}\int_0^t \int_0^s \delta'(B_s-B_r)\, dr
\,ds.
\ee
We now note that the first term on the left is $A(t,B_t)$, and the first term on the right can be expressed in terms of the local time $L(t,x)$. Rearranging, we arrive at

\be
\label{eq:Ito application}
A(t,B_t) - \int_0^t L(s,B_s)
\, dB_s = t+\frac{1}{2}\int_0^t \int_0^s \delta'(B_s-B_r)\, dr
\,ds.
\ee
This formal identity was stated in \cite{rosen2005}. We note, however, that there is some ambiguity, since if we change the definition of $h$ only slightly to
$h(x) := 1_{(0,\ff)}(x)$, change the definition of $A$ to $A(t,x) = \int_0^t 1_{(0,\infty)}(x-B_r) \,dr$, and apply It\^o's formula in the same manner, we lose the $t$ term:
\be
\label{eq:Ito application234}
A(t,B_t) - \int_0^t L(s,B_s)
\, dB_s = \frac{1}{2}\int_0^t \int_0^s \delta'(B_s-B_r)\, dr
\,ds.
\ee
This observation shows that care must be taken, as the two different definitions of $A(t,x)$ agree almost surely. We will see later, in fact, that neither \rrr{eq:Ito application} nor \rrr{eq:Ito application234} is correct, and that the deterministic term required is $\frac{t}{2}$.

Using \eqref{eq:Ito application} as motivation, Rosen \cite{rosen2005}
showed the existence of a process $\alpha_t'(y)$ now known as the {\it derivative
of self-intersection local time} (DSLT) for $B_t$ which is  formally
defined as
\be \label{def:BMversion}
\a := -\int_0^t \int_0^s \dd '(B_s -B_r -y)\, dr\, ds.
\ee
This process was later used in  \cite{hu2009stochastic} and \cite{hu2010central} to prove
central limit theorems for  the $L^2$ and $L^3$ moduli of continuity of Brownian local time.

In
\cite{markowsky2008proof}, it was rigorously proved that almost surely, for all $y$ and $t$,
\be\label{gregs_eqn}
\frac{1}{2}\,\aaa'_t(y) +\frac{1}{2}\sgn(y)t=
 \int_0^t L_s^{B_s-y}\,dB_s - \frac{1}{2}\int_{0}^{t} \sgn
(B_{t}-B_r-y) \, dr \;.\ee
%
%
%
%
%
%
%
%
In particular, if we take $y=0$ and use
$\frac{1}{2}(\sgn(x)+1)$ instead of $1_{[0,\infty)}$ or
$1_{(0,\infty)}$, we obtain

\be
\label{greg_eqn2}
A(t,B_t) - \int_0^t L(s,B_s)
\, dB_s = \frac{t}{2}+\frac{1}{2}\int_0^t \int_0^s \delta'(B_s-B_r)\, dr
\,ds.
\ee
%

\noindent Note that we have the deterministic term $t/2$ instead of the terms $t$ or $0$ which appear in \rrr{eq:Ito application} and \rrr{eq:Ito application234}. This arises from the use of 
$$\frac{1}{2}(\sgn(x)+1)=\frac{1}{2}(1_{[0,\infty)}+1_{(0,\infty)})$$ which is justified by the calculations in the proof of our Tanaka formula in Section 4 below. To see an immediate calculation of this term,
note that the expectations of the two integral terms in \rrr{greg_eqn2} are both 0, by properties of stochastic integrals and the fact that $\dd'$ is an odd distribution. The deterministic term must therefore be $\E[A(t,B_t)]$, and by symmetry we have $$\E[A(t,B_t)] = \frac{1}{2}\E[\int_0^t 1_{[0,\infty)}(B_t-B_r) \,dr + \int_0^t 1_{(-\infty,0)}(B_t-B_r) \,dr] = \frac{1}{2} \int_{0}^{t} dr = t/2.$$

We will refer to \rrr{greg_eqn2} as the {\it Tanaka formula} for $A(t,B_t)$. The method of proof in \cite{markowsky2008proof} also serves to give an alternate proof of
the existence of $\a$ and joint continuity
of $\a + \sgn(y)t$, which Rosen had deduced earlier by other methods. In \cite{markowsky2011},
yet another existence proof for $\aaa_t '(y)$ was given using its Wiener chaos expansion.
Our aim is to extend these results to fractional Brownian motion.

Standard fractional Brownian motion (fBm), with Hurst parameter $H\in(0,1)$, is the unique
centered Gaussian process with covariance function \be
\E(B_t^HB_s^H) = \frac{1}{2}(s^{2H} + t^{2H} - |t-s|^{2H}).\ee
Note that $H=1/2$ gives us a
standard Brownian motion.
In \cite{hu2001}, it was shown that the self-intersection local time of fBm is differentiable in the Meyer-Watanabe sense.
Using Hu's arguments, for $H<2/3$, \cite{yan2008} deduced the existence\footnote{A slightly erroneous bound is
utilized in both \cite{hu2001} and \cite{yan2008}, as well as in other sources, for which we have provided a corrected modification in Appendix \ref{appendix}.} of a process formally defined by
\be \lll{bsl230}
\tilde\alpha_{t}'(0) := -\int_0^t \int_0^s \delta' (B_s^H -B_r^H)s^{2H-1} \, dr\, ds.
\ee
It was also stated (\cite[Eq. (2.4)]{yan2008}) that the identity
\be\lll{mar0} H \,\tilde \aaa'_t(0)= t +
 \int_0^t L_s^{B^H_s}\, dB^H_s - \frac{1}{2}\int_{0}^{t} \sgn
(B^H_{t}-B^H_r) \, dr \ee
holds in the sense of distributions, and this reduces to  \eqref{eq:Ito application} when $H=1/2$. However, we will argue that \rrr{mar0} is incorrect, and in fact that the definition \rrr{bsl230} is also not the proper one needed to study the functionals given on the right side of \rrr{mar0}. To begin with, as stated above, \eqref{greg_eqn2} and the ensuing arguments show that the $t$ term in \eqref{eq:Ito application} should be replaced by $\frac{t}{2}$, and by the same reasoning should in \rrr{mar0} as well. The larger problem lies in the application of It\^o's formula, and the presence of the kernel $s^{2H-1}$ in \rrr{bsl230}. Using the fractional It\^o formula (Theorem \ref{lemma:Ito}) and arguing as in step \rrr{cleopatra}, with $\phi(x) = \frac{1}{2}(\sgn(x)+1)$ in place of $1_{[0,\ff)}(x)$, yields

\be \lll{cleopatra2}
\phi(B_t - B_r) - \phi(0) = \int_r^t \dd(B_s - B_r)\, dB_s + H \int_r^t \dd'(B_s - B_r)\,(s-r)^{2H-1}ds.
\ee
Integrating with respect to $r$ from $0$ to $t$ and switching the order of integration leads to
\be
\label{eq:Ito application2}
A(t,B_t) - \int_0^t L(s,B_s)
\, dB_s = \frac{t}{2}+H \int_0^t \int_0^s \delta'(B_s-B_r)\,(s-r)^{2H-1} dr
\,ds.
\ee
%
%
The more natural version of the DSLT of fBm is therefore revealed to be formally defined by
\be \label{bsl2} \alpha_{t}'(y) := -H \int_{0}^{t}\int_{0}^{s}
\dd'(B^H_s-B^H_r-y) (s-r)^{2H-1} \,dr \,ds.
\ee
That \rrr{bsl2} is the correct definition, rather than \rrr{bsl230}, can be intuitively justified by noting the importance of the set $\{0 \leq r=s \leq t\}$ for any type of intersection local time, rather than the set $\{s=0\}$. However, the singularity of this kernel at the set $\{0 \leq r=s \leq t\}$ makes the existence of \rrr{bsl2} a bit more delicate than that of \rrr{bsl230} for $H < 1/2$. Nonetheless, when $H<2/3$, we will be able to show that such a process exists as a limit in $L^2(\PP)$.  In particular, let $f_1(x)$
denote the standard Gaussian density and set $f_\eps(x) := \frac{1}{\sqrt{\eps}}f_1(\frac{x}{\sqrt{\eps}})$, so that
\be f_\eps(x) = \frac{1}{\sqrt{2\pi\eps}} e^{-\frac{1}{2}x^2/\eps}\quad \text{and}\quad f_\eps'(x) =
\frac{d}{dx} f_\eps(x).\ee
Note that as $\eps\to 0$, $f_\eps(x)$ converges weakly to the Dirac delta distribution, $\dd(x)$. Recall that any $L^2(\PP)$ random variable $F$ measurable with respect to the filtration of $B_s^H$ at time $t$  admits a Wiener chaos expansion (see \cite{nualart1995}) $F=\sum_{n=0}^{\ff} I_n(f_n)$, where $I_n$ represents the multiple Wiener-It\^o integral:

\begin{equation} \label{}
I_n(f_n) = \frac{1}{n!} \int_{0}^{t} \int_{0}^{t} \cdots \int_{0}^{t} f_n(v_1, v_2, \ldots, v_n) dB^H_{v_n} \ldots dB^H_{v_1}.
\end{equation}
We are now ready to state our main results. The process $\alpha_t'(y)$ is rigorously defined as the $L^2(\PP)$ limit, as $\eps\to 0$, of
\be \lll{bsl23}
\alpha_{t,\eps}'(y) := -H\int_0^t \int_0^s f'_\eps (B_s^H -B_r^H
-y)(s-r)^{2H-1} \, dr\, ds.
\ee
The Wiener chaos can be used to show existence in $L^2(\PP)$, and we have the following theorem.

\begin{theorem} \label{mmm}
For $H<2/3$, $\aaa'_t(y)$ exists in $L^2(\PP)$.  In particular, the Wiener chaos decomposition for $\aaa'_t(0)$  is

\be \lll{a1}
\aaa'_{t}(0) = \sum_{m=1}^\ff I_{2m-1}(g(2m-1,t))
\ee
where $g(2m-1,t)\in\H^{\otimes 2m-1}$ is defined by

\bea \lll{a2}
  g(2m-1,t) &=& g(2m-1,t;v_1, \ldots, v_{2m-1})
\\ &=&\nn \frac{(-1)^m}{(m-1)!2^{m-1}\sqrt{2\pi}} \int_{0}^{t}\int_{0}^{s}\frac{\prod_{j=1}^{2m-1}\rho^H_{r,s}(v_j)\,dr\,ds}{(s-r)^{H(2m-1)+1}}
\eea
and
\be\label{explicitform}
\rho^H_{r,s}(x)= \frac{c_H\(\frac{s-x}{|s-x|^{3/2-H}}-\frac{r-x}{|r-x|^{3/2-H}}\)}{2\Gamma(H+1/2) \cos(\frac{\pi}{2}(H+1/2))}.
\ee
\end{theorem}

%
%
Furthermore, we shall prove the Tanaka formula in the following form.

\begin{theorem}\label{tanakaThm} For $0<H<2/3$,  the following equality holds in $L^2(\PP)$  for all y and t:
\be\lll{mar} H \,\aaa'_t(y) + \frac{1}{2}\sgn(y)t =
 \int_0^t L_s^{B^H_s-y}\, dB^H_s - \frac{1}{2}\int_{0}^{t} \sgn
(B^H_{t}-B^H_r-y) \, dr \;.\ee
\end{theorem}


As indicated earlier, this can easily be transformed into a formula containing $A(t,B_t)$ if desired. The rest of the paper is organized as follows. In  Section \ref{sec:existence2}, we prove the existence of   $\aa'_t(0)$ in $L^2(\PP)$ for $H<2/3$ and discuss a conjecture for the behavior when $H\ge 2/3$.
 In
 Section \ref{sec:chaos} we prove the
Wiener chaos expansion which also extends the existence of DSLT to all $y\in\R$. Finally, in Section \ref{sec:tanaka} we prove the Tanaka formula.
For completeness, in the appendices
we have included a quick review of some tools from Malliavin calculus that we have used. One of these tools, a Fubini theorem for fractional Brownian integrals, is of independent interest as it generalizes, to Hida distributions, similar results found in \cite{cheridito2005} and \cite{mishura2008}.

As a side note, we mention that it also may be considered natural to define the DSLT of fBm in a different manner, namely with the kernel $(s-r)^{2H-1}$ removed in \rrr{bsl2}. This other version is more natural for an occupation times formula, as well as for differentiation in the space variable $y$. An interested reader is referred to \cite{jung2012holder} for details. However, as mentioned above, our primary interest here lies in the Tanaka formula \rrr{mar}, which does require the kernel to be present.

\section{Existence of $\aaa'_t(0)$}\label{sec:existence2}

Fix $0<H<2/3$. One does not need the Wiener chaos to show existence of the DSLT when $y=0$. In this case, an easier proof
can be deduced by the arguments of
\cite{yan2008}.
This proof presents some points of interest, so (without full rigor) we sketch it now.

The key lies in computing
$\E[(\aaa_{t,\eps}'(0))^2]$. In the sequel, $K$ denotes a positive constant which may change from line to line.
We start by expressing
$f_\eps(x)$ in a convenient form using the Fourier identity $f_1(x) = \frac{1}{2\pi}\int_{\mathbb{R}} e^{ipx}e^{-p^2/2}dp$. This gives
\be \label{}\nn
f_\eps(x) = 
\frac{1}{2\pi\sqrt{\eps}}\int_{\mathbb{R}}e^{ipx/\sqrt{\eps}} e^{-p^2/2}\,dp = \frac{1}{2\pi}\int_{\mathbb{R}}e^{ipx} e^{-\eps p^2/2}\,dp,
\ee
whence
\be \label{}
f'_\eps(x) = \frac{i}{2\pi}\int_{\mathbb{R}} p e^{ipx} e^{-\eps p^2/2}\,dp.
\ee

Now, for $\DD_t=\{0\leq
r\leq s \leq t\}$,
\bea \label{exp} &&\E[(\aaa_{t,\eps}'(0))^2] =\\
\nn&& K \int_{\DD_t^2} \int_{(p,q)\in\mathbb{R}^2} pq
e^{-\eps(p^2+q^2)/2}e^{-\mbox{Var}\( p (B^H_{s}-B^H_r)+
q(B^H_{s'}-B^H_{r'})\)/2}\\
\nn&& \qquad \times (s-r)^{2H-1}(s'-r')^{2H-1} \,dp\,dq\,dr\,dr'\,ds\,ds'.
\eea
We will show that this can be bounded uniformly in $\eps$. Using standard notation from the literature (see for example \cite{biagini2008}), let

\bea \label{tor}
\lambda &:=& \mbox{Var}(B^H_{s}-B^H_r) = |s-r|^{2H}
\\ \nn  \ppp &:=& \mbox{Var}(B^H_{s'}-B^H_{r'}) = |s'-r'|^{2H}
\\ \nn  \mu &:=& \mbox{Cov}(B^H_{s}-B^H_r,B^H_{s'}-B^H_{r'}) \\
\nn&=& \frac{1}{2}(|s'-r|^{2H} + |s-r'|^{2H} - |s'-s|^{2H} - |r'-r|^{2H}).
\eea
With this notation, we have
\bea \label{exp2}
 && \E[(\aaa_{t,\eps}'(0))^2] = K \int_{\DD_t^2} (s-r)^{2H-1}(s'-r')^{2H-1} \\
\nn&& \times \int_{\mathbb{R}^2} pq
e^{-\eps(p^2+q^2)/2}e^{-(p^2\lambda +2pq \mu + q^2 \ppp)/2} \,dp\,dq\,dr\,dr'\,ds\,ds'.
\eea
Isolating the $dq$ integral we have:
\bea \label{}
&& \int_\R qe^{-pq\mu}e^{-q^2(\ppp+\eps)/2} \,dq \\ \nn && =
e^{\frac{p^2\mu^2}{2(\ppp+\eps)}}\int_\R
qe^{-(q+\frac{p\mu}{\ppp+\eps})^2(\ppp+\eps)/2} \,dq
\\ \nn && = e^{\frac{p^2\mu^2}{2(\ppp+\eps)}}\left[\int_\R qe^{-q^2(\ppp+\eps)/2}
\,dq - \frac{p\mu}{(\ppp+\eps)} \int_\R
e^{-q^2(\ppp+\eps)/2} \,dq\right].
\eea
The first term on the right side is an integral of an odd function, and thus
vanishes. The second integral on the right side converges as $\eps\to 0$ to
\be \label{}\nn
Ke^{p^2\mu^2/(2\ppp)}\frac{p\mu}{\ppp^{3/2}}.
\ee
At $\eps=0$, we therefore have  for the $dp$ integral,
\be \label{jkm}\nn
\frac{K\mu}{\ppp^{3/2}}\int_\R p^2 e^{-\frac{p^2}{2\ppp}(\lambda \ppp - \mu^2)} \,dp = \frac{K\mu}{(\lambda \ppp - \mu^2)^{3/2}}.
\ee
Thus, we have reduced the problem to determining the integrability, over $\DD_t^2$, of
\be
\frac{K\mu
(s-r)^{2H-1}(s'-r')^{2H-1}}{(\la \ppp - \mu^2)^{3/2}}.
\ee
The existence of $\aaa'_t(0)$
is therefore proved by invoking the following lemma, which is proved in the appendix.

\bl \lll{lab} If $H<2/3$, then

\be \lll{t3}
\int_{\DD_t^2}\frac{\mu (s-r)^{2H-1}(s'-r')^{2H-1}}{(\la \ppp - \mu^2)^{3/2}}  \,dr\,dr'\,ds\,ds'<\ff.
\ee

\el
A few remarks are in order. First, something close to Lemma \ref{lab} was proved in \cite{yan2008},
but they had a factor of $s^{2H-1}$ where we have $(s-r)^{2H-1}$. As indicated earlier, this resulted from incorrectly applying the fractional It\^o formula to $B^H_s$ as opposed to $B^H_s-B^H_r$.

Next we note that in \cite{rosen2005}, Rosen states ``The
(DSLT of Brownian motion) in $\mathbb{R}^1$, in a certain sense, is even
more singular than self-intersection local time in $\mathbb{R}^2$.''
If we believe that the critical Hurst parameter $H_c$ for  the DSLT to exist in $L^2$ is $2/3$, as seems likely, then
Rosen's statement would be supported by the fact that $2/3$ is less than the
critical Hurst parameter  for the self-intersection local time of planar fBm (which is most likely $\tilde H_c=3/4$, see \cite{rosen1987, hu2005}).

Above the critical parameter $H_c$, the behavior of $\aaa'_{t,\eps}(0)$ as $\eps \lar 0$ is also of interest. One would expect a central
limit theorem to exist, along the lines of
Theorem 2 in \cite{hu2005} or Theorem 1 in \cite{markowsky2008renormalization},
but this remains unproved. In particular, it seems as though the techniques developed in \cite{hu2005} should apply,
especially since the Wiener chaos expansion for $\aaa'_\eps$ is readily computed (see Section \ref{sec:chaos}),
but the presence of the derivative seems to complicate matters. Nevertheless, we venture the following conjecture.

\vski
{\it \bf Conjecture:} {\it
\begin{itemize}
\item The critical parameter is $H_c=2/3$. At $H_c$, $\frac{1}{\log(1/\eps)^\ga} \aaa'_{t,\eps}(0)$ converges in distribution to a normal law for some $\ga>0$.
\item For $H>H_c$, $\eps^{-\ga(H)} \aaa'_{t,\eps}(0)$ converges in distribution to a normal law for some function
$\ga(H)>0$ which is linear in $1/H$ and for which $\ga(2/3)=0$.
\end{itemize}
}
\vski
\noindent This would mirror the behavior of the intersection local time as seen in \cite{hu2005} (and, to a lesser extent, \cite{markowsky2008renormalization}, which dealt with stable processes), in which $L^2$ convergence occurred for $\aaa_\eps$ for $H < H_c$, $\aaa_\eps$ diverged logarithmically in $\eps$ for $H=H_c$, and $\aaa_\eps$ diverged polynomially in $1/\eps$ for $H>H_c$. We should mention that this problem may prove to be very difficult: to our knowledge, no such central limit theorem has yet been proved even for intersection local time in two dimensions at $H_c= 3/4$.

\section{The Wiener chaos and existence of $\aa'_t(y)$.}\label{sec:chaos}

In this section we prove Theorem \ref{mmm}.  We break this up into two parts by first deriving the Wiener chaos expansion for $\aaa'_t(0)$.  We then  adapt the arguments used in obtaining the Wiener chaos to
show existence, in $L^2(\PP)$, of
$$
\aa'_t(y)= -H\lim_{\eps_\to 0}\int_0^t \int_0^s f'_\eps (B_s^H -B_r^H
-y)(s-r)^{2H-1} \, dr\, ds
$$
 for all $y\in\R$. To reduce notation, in the sequel we write $\H=L^2(\R)$. Our arguments in this section closely parallel those in \cite{hu2005}.

\begin{proof}[Proof of the Wiener chaos expansion]
For the explicit form of $\rho^H_{r,s}$ given in \eqref{explicitform} we refer the reader to Example A1 in \cite{elliot2003}.
For \eqref{a1} and \eqref{a2} we use fractional white noise theory (see Appendix \ref{app:fwn}) to write
$$B^H_t:=\lan \omega, \rho^H_{0,t}\ran$$
 so that in particular, its Malliavin derivative is
$DB^H_t=\rho^H_{0,t}$.

It is not hard to see that $f'_\eps(B^H_s-B^H_r)$ is in $\cap_{k\in\N}\mathbb{D}^{k,2}$,
thus by Stroock's formula, the $n$-th integrand in the chaos expansion of $\aaa'_{t,\eps}(0)$ is given by

\bea \lll{}
&&\frac{1}{n!}\int_{0}^{t}\int_{0}^{s} (s-r)^{2H-1}\E[D^n_{v_1,\ldots,v_n}f_\eps '(B^H_s-B^H_r)]\,dr\,ds \\
&=&\nn \frac{1}{n!}\int_{0}^{t}\int_{0}^{s} (s-r)^{2H-1}\E[(\frac{d^{n+1}}{dx^{n+1}} f_\eps)(B^H_s-B^H_r)]\prod_{j=1}^n \rho^H_{r,s}(v_j)\,dr\,ds.
\eea

Similar to \eqref{bsl23}, we write
\be \label{}
\frac{d^{n}}{dx^{n}}f_\eps(x) = \frac{i^{n}}{2\pi}\int_{\mathbb{R}}e^{ipx} p^{n}e^{-\eps p^2/2}\,dp.
\ee
Thus,
\bea \label{heyoh}\nn
&& \E[(\frac{d^{n+1}}{dx^{n+1}} f_\eps)(B^H_s-B^H_r)] = \frac{i^{n+1}}{2\pi}\int_{\mathbb{R}}\E[e^{ip(B^H_s-B^H_r)}] p^{n+1}e^{-\eps p^2/2}\,dp
\\ \nn && \hspace{1cm} = \frac{i^{n+1}}{2\pi}\int_{\mathbb{R}}p^{n+1}e^{-((s-r)^{2H}+\eps) p^2/2}\,dp
\\ \nn && \hspace{1cm} = \frac{i^{n+1}}{2\pi((s-r)^{2H}+\eps)^{(n/2)+1}}\int_{\mathbb{R}}p^{n+1}e^{- p^2/2}\,dp
\\  \nn && \hspace{1cm} = \frac{i^{n+1}\sqrt{2\pi}}{2\pi((s-r)^{2H}+\eps)^{(n/2)+1}}\frac{(n+1)!}{2^{(n+1)/2}((n+1)/2)!}
\\  && \hspace{1cm} = \frac{(-1)^{(n+1)/2}}{\sqrt{2\pi}((s-r)^{2H}+\eps)^{(n/2)+1}}\frac{n!}{2^{(n-1)/2}((n-1)/2)!}
.
\eea
if $n+1$ is even, and $0$ if $n+1$ is odd. Setting $n=2m-1$, it follows that the chaos expansion for $\aaa'_{t,\eps}(0)$ is
\bea \label{ni}
&&\nn \sum_{m=1}^\ff I_{2m-1}
\left( \frac{(-1)^m}{(m-1)!2^{m-1}\sqrt{2\pi}} \int_{0}^{t}\int_{0}^{s}\frac{(s-r)^{2H-1}}{(\eps + (s-r)^{2H})^{(2m+1)/2}} \prod_{j=1}^{2m-1}\rho^H_{r,s}(v_j)\,dr\,ds\right).
\\ && \hspace{1cm}
\eea

We now need to show that as $\eps \to 0$, the above converges in $L^2(\PP)$ to
\rrr{a1}. We will apply the following lemma adapted from \cite{nualart1992}, which is a consequence of the dominated convergence theorem.
\begin{lemma} \label{lem}
Let $F_\eps$ be a family of $L^2(\PP)$ random variables with chaos expansions
$F_\eps = \sum_{n=0}^{\ff}I_n(f_n^\eps)$. If for each $n$, $f_n^\eps$ converges in $ \H^{\otimes n}$ to $f_n$ as $\eps \lar 0$, and if

\begin{equation} \label{psb}
\sum_{n=0}^{\ff} \sup_\eps \E[I_n(f_n^\eps)^2] = \sum_{n=0}^{\ff} \sup_\eps\{n!||f_n^\eps||^2_{\H^{\otimes n}}\} < \ff,
\end{equation}
then $F_\eps$ converges in $L^2(\PP)$ to $F=\sum_{n=0}^{\ff}I_n(f_n)$ as $\eps \lar 0$.
\end{lemma}

We note that this argument has also been used in \cite{hu2005} and \cite{markowsky2011}.
To apply the lemma here, we calculate the $L^2(\PP)$-norms of
the chaos expansions and show they are bounded uniformly in
$\eps$. Recall that $\DD_t=\{0\leq
r\leq s \leq t\}$. If we let $g(2m-1,t,\eps)$ be the integrand of
$I_{2m-1}$ in \rrr{ni}, we have

\begin{equation} \label{opps}
\begin{split}
& \E [I_{2m-1}(g(2m-1,t,\eps))^2] =(2m-1)!||g(2m-1,t,\eps)||^2_{\H^{\otimes 2m-1}} \\
& = \frac{(2m-1)!}{2\pi[(m-1)!]^2 2^{2m-2}} \int_{\DD_t^2} \frac{(s-r)^{2H-1}(s'-r')^{2H-1}}{(\eps + (s-r)^{2H})^{(2m+1)/2}(\eps + (s'-r')^{2H})^{(2m+1)/2}} \\
& \qquad \qquad \qquad \times\left(\int_{\mathbb{R}^{2m-1}} \prod_{j=1}^{2m-1} \langle \rho^H_{r,s}(v_j),\rho^H{r',s'}(v_j) \rangle_{\H}\, dv_j\right) \,dr\,dr'\,ds\,ds'.
\end{split}
\end{equation}
Maximizing by setting $\eps=0$ and using (A.10) in \cite{elliot2003} and the notation in \rrr{tor}, this is equal to
\be\label{toys}
\frac{m(2m)!}{\pi (m!)^2 2^{2m}} \int_{\DD_t^2} \frac{(s-r)^{2H-1}(s'-r')^{2H-1}\mu^{2m-1}}{\la^{m+1/2}\rho^{m+1/2}} \,dr\,dr'\,ds\,ds'.
\ee

Let $\ga = \mu^2/(\la\rho)$. The $L^2(\PP)$-norm of \rrr{ni} is then

\be\label{cloys}
\frac{1}{\pi} \int_{\DD_t^2} \left(\sum_{m=1}^\ff \frac{m(2m)!\ga^m}{(m!)^22^{2m}}\right)\frac{(s-r)^{2H-1}(s'-r')^{2H-1}\,dr\,dr'\,ds\,ds'}{\mu \sqrt{\la \rho}}.
\ee
However, by the generalized binomial theorem,

\be \lll{trilo}
\frac{\ga}{2(1-\ga)^{3/2}}=\sum_{m=1}^\ff \frac{m(2m)!\ga^m}{(m!)^22^{2m}}.
\ee
Thus, the $L^2(\PP)$-norm of \rrr{ni} is

\be \lll{}
\nn\frac{1}{2\pi} \int_{\DD_t^2} \frac{\ga(s-r)^{2H-1}(s'-r')^{2H-1}}{(1-\ga)^{3/2}\mu \sqrt{\la \rho}} =
\frac{1}{2\pi} \int_{\DD_t^2} \frac{\mu(s-r)^{2H-1}(s'-r')^{2H-1}}{(\la \ppp - \mu^2)^{3/2}}.
\ee
By Lemma \ref{lab}, this is finite if $H<2/3$.
\end{proof}


{\bf Remark:}  One can also think of fBm as an isonormal Gaussian process $W:\H'\to L^2(\PP)$ on the space
$$\H':= L^2_H(\R):=\{f:M_Hf\in L^2(\R)\}$$
which may contain distributions. Then $\rho^H_{r,s}\equiv M_H1_{[r,s]}$ and
\bea
\langle 1_{[0,t]},1_{[0,s]}\rangle_{\H'}&:=&\lan \rho^H_{0,t},\rho^H_{0,t}\ran_{L^2(\R)}\\
&=&\frac{1}{2}(t^{2H} + s^{2H} - |t-s|^{2H}).
\nn\eea
Using this so called ``twisted'' inner product Hilbert space, the  isonormal Gaussian process gives $B^H_t=\lan \omega, 1_{[0,t]}\ran_{tw}$
 so that
$DB^H_t=1_{[0,t]}$ (In \cite[Sec. 3]{hu2005}, the twisted inner product space was utilized without reference to $M_H$ or $\rho^H_{r,s}$, and the reader can check that our argument above carries over easily into that setting).  Comparing this with the above, we see that the twisted inner product incorporates the operation of $M_H$ in $\lan\omega,M_H f\ran$ into
$\lan \omega, f\ran_{tw}$.
When $f$ is a step function, essentially nothing but notation has changed, which
can be verified by the use of the $D^M$ operator in Example 6.4 from \cite{elliot2003}.
One may then write

\begin{equation} \label{}
\int_{0}^{t}\int_{0}^{s}\frac{(s-r)^{2H-1}}{( (s-r)^{2H})^{(2m+1)/2}} \prod_{j=1}^{2m-1}1_{[r,s]}(v_j)\,dr\,ds
= \int_{\overline{v}}^{t}\int_{0}^{\underline{v}}\frac{(s-r)^{2H-1}}{( (s-r)^{2H})^{(2m+1)/2}} \,dr\,ds,
\end{equation}
where
\bea \label{tear}
&&\nn \overline{v}=v_1\vee \ldots \vee v_{2m-1},
\\ \nn && \underline{v}=v_1 \wedge \ldots \wedge v_{2m-1}.
\eea
It is then straightforward to verify that Theorem \ref{mmm} simplifies to the chaos expansion given in \cite{markowsky2011} in the case $H=1/2$.

\vski

We now use the methods in the above proof to show $L^2(\PP)$ convergence for $\aaa'_{t,\eps}(y)$ as $\eps\to 0$.

\begin{proposition} \label{mmm2}
For $H<2/3$ and any $y \in \reals$, $\aaa'_t(y)$ is in $L^2(\PP)$.
\end{proposition}

\begin{proof}
We may follow the proof of Theorem \ref{mmm}, except that in place of \rrr{heyoh} we have

\bea \label{heyoh2}\nn
&& \Big|\E[(\frac{d^{n+1}}{dx^{n+1}} f_\eps)(B^H_s-B^H_r-y)]\Big| = \Big|\frac{i^{n+1}}{2\pi}\int_{\mathbb{R}}
\E[e^{ip(B^H_s-B^H_r)}] e^{ipy} p^{n+1}e^{-\eps p^2/2}\,dp\Big|
\\ && \hspace{2cm} \leq \frac{1}{2\pi((s-r)^{2H}+\eps)^{(n/2)+1}}\int_{\mathbb{R}}|p|^{n+1}e^{- p^2/2}\,dp.
\eea
We aim to apply Lemma \ref{lem} again. The arguments from  Theorem \ref{mmm} show that the sum of the odd terms in \rrr{psb} converges.
However, we can no longer argue that the even terms are 0, as we did before.
Instead, we must use the identity $\int_{\mathbb{R}}|p|^{n+1}e^{- p^2/2}\,dp = 2^{n/2}(n/2)!$, valid for even $n$.

Replacing \rrr{opps}, we then have

\begin{equation} \label{opps2}
\begin{split}
& \E [I_{2m}(g(2m,t,\eps))^2] = (2m)!||g(2m,t,\eps)||^2_2 \\
& = \frac{(2m)!(m!)^2 2^{2m}}{(2m!)^2} \int_{\DD_t^2} \frac{(s-r)^{2H-1}(s'-r')^{2H-1}}{(\eps + (s-r)^{2H})^{m+1}(\eps + (s'-r')^{2H})^{m+1}} \\
& \qquad \qquad \qquad \times\left(\int_{\mathbb{R}^{2m}} \prod_{j=1}^{2m} \langle \rho^H_{r,s}(v_j),\rho^H_{r',s'}(v_j) \rangle_\H dv_j\right) \,dr\,dr'\,ds\,ds'.
\end{split}
\end{equation}
We proceed through steps \rrr{toys} and \rrr{cloys}, setting $\eps = 0$ and $\ga = \mu^2/(\la\rho)$, in order to reach a bound on the even terms of

\be \lll{cloys2}
\frac{1}{\pi} \int_{\DD_t^2} \left(\sum_{m=1}^\ff \frac{(m!)^2 2^{2m}\ga^m}{(2m)!}\right)\frac{(s-r)^{2H-1}(s'-r')^{2H-1}\,dr\,dr'\,ds\,ds'}{\la \rho}.
\ee
We recall \rrr{trilo}, and note that it is easy to verify using Stirling's approximation that

\begin{equation} \label{}
\sup_m \Big(\frac{(m!)^2 2^{2m}}{(2m)!}\Big)\Big(\frac{m(2m)!}{(m!)^22^{2m}}\Big)^{-1} = K < \ff.
\end{equation}
In reference to \rrr{cloys2} we therefore have for $0\leq \ga < 1$

\begin{equation} \label{}
\sum_{m=1}^\ff \frac{(m!)^2 2^{2m}\ga^m}{(2m)!} \leq K \sum_{m=1}^\ff \frac{m(2m)!\ga^m}{(m!)^22^{2m}} = \frac{K \ga}{2(1-\ga)^{3/2}} \leq \frac{K \sqrt{\ga}}{2(1-\ga)^{3/2}}.
\end{equation}
%
%
Inserting this bound and the expression for $\ga$ into \rrr{cloys2} gives a bound of

\begin{equation} \label{}
\begin{split}
\frac{K}{2\pi} &\int_{\DD_t^2} \frac{\mu (s-r)^{2H-1}(s'-r')^{2H-1}\,dr\,dr'\,ds\,ds'}{(1-\frac{\mu^2}{\la \rho})^{3/2}\la^{3/2} \rho^{3/2}}\\
& = \frac{K}{2\pi} \int_{\DD_t^2}\frac{\mu (s-r)^{2H-1}(s'-r')^{2H-1}}{(\la \ppp - \mu^2)^{3/2}}  \,dr\,dr'\,ds\,ds'<\ff.
\end{split}
\end{equation}
By Lemma \ref{lab}, this is finite for $H<2/3$.
\end{proof}

\section{The Tanaka formula}\label{sec:tanaka}

%

\begin{proof}[Proof of Theorem \ref{tanakaThm}]
Let $f_\eps$ be defined as in \eqref{bsl23} and let
\be
F_\eps(x) = \int_0^x f_\eps(u) \,du = \int_0^{x/\eps} f(u) \,du.
\ee
By Theorem \ref{lemma:Ito} (fractional It\^o formula), we have
\bea
&&F_\eps(B^H_t - B^H_r - y) - F_\eps(-y) \\
&=& \nn\int_r^t f_\eps(B^H_s - B^H_r -y)\, dB^H_s + H \int_r^t f'_\eps(B^H_s - B^H_r - y)(s-r)^{2H-1}\, ds
\eea
Integrating with respect to $r$ from $0$ to $t$ gives

\bea\label{eq:Ito application2b}
&&\int_0^t F_\eps(B^H_t - B^H_r - y)\, dr -tF_\eps(-y) =\\
&& \nn\int_0^t\int_r^t  f_\eps(B^H_s - B^H_r -y)\, dB^H_s\,dr + H \int_0^t\int_r^t f'_\eps(B^H_s - B^H_r - y)(s-r)^{2H-1}\, ds \,dr.
\eea
Note that $F_\eps(-y) \to -\frac{1}{2} \sgn(y)$ as $\eps\to 0$. Also, $|F_\eps(\cdot)|\le 1/2$ for all $\eps>0$,
so by dominated convergence the integral term on the left side approaches $$\frac{1}{2}\int_0^t \sgn(B^H_t - B^H_r - y)\,dr$$ as $\eps\to 0$.
We now want to apply Fubini's theorem for fractional Brownian integration (Theorem \ref{lemma:Fubini}) to the first term on the left side to get
\be\label{afterfubini}
\int_0^t\int_0^s f_\eps(B^H_s - B^H_r - y) \,dr \,dB^H_s.
\ee

To justify Fubini and get \eqref{afterfubini}, we use the following Wiener-I\^o chaos expansion obtained from Stroock's formula:
\be
f_\eps(B^H_s - B^H_r - y)=\sum_{n\ge 0}I_n\(\frac{1}{n!}\E[(\frac{d^{n}}{dx^{n}} f_\eps)(B^H_s-B^H_r-y)] (\rho^H_{r,s})^{\otimes n}\).
\ee
The Hermite chaos refines the $n$th Wiener-It\^o chaos in terms of the Hermite orthonormal basis of $\H^{\otimes n}$.
In particular, the coefficient $c_\bb(s,r)$, with $|\bb|=n$, in the Hermite chaos expansion
of $f_\eps(B^H_s - B^H_r - y)$ is given by
\be
c_\bb(s,r)=\frac{1}{n!}\E[(\frac{d^{n}}{dx^{n}} f_\eps)(B^H_s-B^H_r-y)] \lan\xi^{\odot\bb}, (\rho^H_{r,s})^{\otimes n} \ran_{\H^{\otimes n}}.
\ee
Using \eqref{heyoh2} and the fact that $M_H\xi_k$ is bounded (see Lemma 4.1 of \cite{elliot2003}), one can now easily verify the conditions of Theorem \ref{lemma:Fubini} for $f_\eps(B^H_s - B^H_r - y)$ for $\eps>0$.

For fixed $s$, as $\eps \lar 0$ the inner integral in \eqref{afterfubini} converges to
$L_s^{B^H_s - y}$  in $(\SS)^*$ by  Proposition 10.1.13 in
\cite{biagini2008}. Now, if $L_s^{B^H_s - y}\diamond W^H(s)$  is integrable in $(\SS)^*$, then by the arguments of Proposition 8.1 in \cite{hida1993white},
\eqref{afterfubini} converges in $(\SS)^*$ to $ \int_0^t L_s^{B^H_s - y}\diamond W^H(s)\, ds$.
In other words, the equality in \eqref{mar} is valid as long as one side or the other is in $(\SS)^*$.  However, by Proposition \ref{mmm2}, $\aa_t'(y)\in L^2(\PP)$ , for $H<2/3$.  Thus
for such $H$, \eqref{mar}
holds in $L^2(\PP)$. This completes the proof of Theorem \ref{tanakaThm}.
\end{proof}

{\bf Remark:} One-dimensional fBm has an $L^2(\PP)$ local time for any $0<H<1$ (see \cite{biagini2008}),
but it is not clear whether or not  $ \int_0^t L_s^{B^H_s-y}\, dB^H_s$ is in $L^2(\PP)$ for $H>2/3$. As stated in the conjecture of Section \ref{sec:existence2}, we suspect that it is not. Of course,
a positive answer to the conjecture does not rule out the possibility that
\be \label{openprob}
 \int_0^t L_s^{B^H_s-y}\, dB^H_s\in (\SS)^* \quad\text{ for } H>2/3.
 \ee
 If \eqref{openprob} were indeed true, then the DSLT of fBm would also be well-defined in $(\SS)^*$ for all $H\in(0,1)$.
For $H<2/3$, another open problem is to prove joint continuity, in $y$ and $t$, of
$\a + \sgn(y)t$. One approach to proving this is to use the explicit chaos expansion for this integral (see Theorem \ref{mmm}) combined with Definition \ref{def:WIS integral}.

\appendix

\section{A Fubini theorem for the WIS integral}\label{sec:fbm integration}

In this appendix we prove a Fubini theorem for the WIS integral.  We note that there are several different ways one can integrate with respect to
fBm as can be seen in \cite{biagini2008} or \cite{mishura2008}.  The WIS integral is based on the fractional white noise theory introduced in \cite{elliot2003} (see also \cite{biagini2004}, \cite[Ch.
4]{biagini2008}).


In an effort to be somewhat self-contained,  we first
summarize some results concerning white noise and the WIS
integral. For more details we refer the reader to \cite{hida1993white, holden1996,elliot2003,biagini2004}. The only new result in this appendix is Theorem
\ref{lemma:Fubini} which is
standard for other constructions of stochastic integrals with respect to fBm (cf. \cite[Theorem
3.7]{cheridito2005}, \cite[Theorem 1.13.1]{mishura2008}); however,
we have found no reference for such results with respect to Hida distributions and the WIS
integral. Theorem \ref{lemma:Fubini}  follows easily once one has the right definitions.  Its formulation
may be of interest in its own right, but the main
reason for its presentation here is due to its role in proving the Tanaka formula.

\subsection{Classical white noise theory}

Let $\Lambda=\N^{\N}_0$ be the set of multi-indices with finite support, $\SS(\R)$ be the Schwartz
space, and $\SS'(\R)$ be the space of tempered distributions. By the
Bochner-Minlos theorem (see for example, Appendix A of \cite{biagini2008}) there is a probability measure $\PP$ on the
Borel $\sigma$-field of $\SS'$ satisfying
 \be\label{eqn:Bochner}
\int_{\SS'(\R)}\exp(i\lan \omega,f\ran)\, d\PP(\omega) =
\exp\(-\frac{1}{2}\|f\|^2_{L^2(\R)}\), \quad f\in\SS(\R).
 \ee
This measure on $\omega\in \SS'(\R)$ satisfies, for all $f\in\SS(\R)$,
\be \label{isometry}
\E\lan \omega,f\ran =0 \ \text{ and }\ \E \lan \omega,f\ran^2 = \|f\|^2_{L^2(\R)}.
\ee
If $g_n\in\SS(\R)$  converges to $1_{[0,t]}$ in $L^2(\R)$, we define

\be\label{S_to_L2}
\lan\omega,1_{[0,t]}\ran:=\lim_{n\to\ff}\lan \omega, g_n\rangle
\ee
as a limit in $L^2(\PP)$.
The family $\(1_{[0,t]}\)_{t\ge 0}$ are functionals mapping $\SS'(\R)$  to $\R$,
and can be identified with random variables having characteristic functions $\phi(\nu)=\exp\(-\frac{1}{2}(t\nu)^2\)$.  Choosing a continuous version of this family gives us Brownian motion.

 Recall
that the {\it Hermite polynomials} given by
 \be h_n(x) :=
(-1)^ne^{x^2/2}\frac{d^n}{dx^n}(e^{-x^2/2}), \quad n=0,1,2,\ldots
 \ee
are orthogonal with respect to the standard Gaussian measure on $\R$.  Thus multiplying by the Gaussian density and choosing a convenient normalization gives us an orthonormal basis for $L^2(\R)$ which live in $\SS(\R)$,
 \be
\xi_n(x):=\pi^{-1/4}((n-1)!)^{-1/2}h_{n-1}(\sqrt{2}x)e^{-x^2/2},
\quad n=1,2,3,\ldots,
 \ee
 called  the {\it Hermite functions.}\footnotemark\footnotetext{One may substitute any orthonormal basis of $L^2(\R)$ whose elements  possess decay properties such that Lemma 4.1 of \cite{elliot2003} holds. See also Theorem 3.1 in \cite{ito1951multiple}.}

For $\bb=(\bb_1,\ldots,\bb_n)\in\Lambda$,
we define
 \be
\HH_{\bb}(\omega):=h_{\bb_1}(\lan\omega,\xi_1\ran)h_{\bb_1}(\lan\omega,\xi_2\ran)\cdots
h_{\bb_n}(\lan\omega,\xi_n\ran).
 \ee
Every $F\in L^2(\PP)$ has a representation in terms of the $\HH_{\bb}$:
\be\label{wcII}
F(\omega) = \sum_{\bb\in\Lambda}c_\bb\HH_\bb(\omega)
\ee
where the series converges in $L^2(\PP)$.  Moreover, one has a Parseval-type identity $$\E F^2=\sum_{\bb\in\Lambda} \bb!c_\bb^2.$$

Representation \eqref{wcII} for $F$ is called the Hermite chaos expansion, and it is related to the Wiener-It\^o  chaos expansion\footnotemark\footnotetext{For a full treatment
of multiple Wiener-It\^o integrals and their related chaos expansions, see \cite{nualart1995}.} via the following formula which follows from \cite{ito1951multiple}:
\be\label{itos_mwi}
\int_{\R^n}\xi^{\odot\bb}dB^{\odot n} = \HH_\bb(\omega).
\ee
Here $\odot$ denotes the symmetrized tensor product, and $\xi^{\odot\bb}:=\xi_1^{\odot \bb_1}\odot\cdots\odot\xi_n^{\odot \bb_n}$.
In particular, the Hermite chaos is a way of writing the $n$th Wiener-It\^o chaos in terms of $n$-fold products of Hermite functions, which form orthonormal bases of $L^2(\R^n)$ (see \cite[pg. 30]{holden1996}).

The main reason for using chaos expansions in terms of Hermite polynomials instead of multiple Wiener-It\^o integrals is that the Hermite expansions can be used to generalize
  $L^2(\PP)$ to the space of Hida distributions.
Let
 \be
 (2\N)^\gamma:=(2\cdot1)^{\gamma_1}(2\cdot2)^{\gamma_2}\cdots(2\cdot
 n)^{\gamma_n}\quad \text{for
 }\gamma=(\ga_1,\ldots,\ga_n)\in\Lambda.
 \ee
\begin{definition}[Hida test functions and distributions]\label{def:Hida space}
Given the probability measure $\PP$ on $\SS'$, the Hida test
function space $(\SS)$ is the set of all $\psi\in L^2(\PP)$ given by
$$\psi(\omega) = \sum_{\bb\in\Lambda}a_\bb\HH_\bb(\omega)$$
where $a_\bb$ satisfy
$$ \sum_{\beta\in\Lambda}a_\bb^2\bb!(2\N)^{k\bb}\quad\text{for all
}k=1,2,\ldots$$
The Hida distribution space $(\SS)^*$ is the set of all formal
expansions
$$F(\omega) = \sum_{\bb\in\Lambda}b_\bb\HH_\bb(\omega)$$
where $b_\bb$ satisfy
$$ \sum_{\beta\in\Lambda}b_\bb^2\bb!(2\N)^{-q\bb}\quad\text{for some
}q\in\R.$$

\end{definition}

It was shown in \cite{zhang1992characterizations} that $(\SS)^*$ is
the dual of $(\SS)$.
Moreover, by Corollary 2.3.8 of \cite{holden1996}, $$(\SS)\subset
L^2(\PP)\subset (\SS)^*.$$ This should be thought of as analogous to
the triplet $\SS\subset L^2(\R)\subset \SS^*$. Note that if $F\in
L^2(\PP)$, then $$\lan\lan F,\psi\ran\ran=\lan
F,\psi\ran_{L^2(\PP)}=\E(F\psi).$$
Thus,
for $\psi=\sum_{\beta\in\Lambda} a_\beta\HH_\bb \in(\SS)$ and
$F=\sum_{\beta\in\Lambda} b_\beta\HH_\bb\in(\SS)^*$, the duality inherited from $L^2(\PP)$ is given by
\be\label{hida_inner_product}
\lan\langle F, \psi \ran\rangle := \sum_{\beta\in\Lambda} \beta !
a_\beta b_\beta. \ee

\begin{lemma}\label{lem:hidaintegral}
Suppose $F$ is an $(\SS)^*$-valued function on a $\sigma$-finite
measure space $(X,\BB,\nu)$ and that
 \be\label{A9}
 \lan\lan F(x), \psi
\ran\ran \in L^1(X,\nu(dx)) \quad \forall\ \psi \in (\SS).
 \ee
 Then there
is a unique $G$ in $(\SS)^*$ such that \be\label{A10} \lan\lan G,
\psi \ran\ran = \int_X\lan\lan F(x), \psi \ran\ran \,\nu(dx) \quad
\forall\ \psi \in (\SS). \ee We write$\int_X F(x)\,\nu(dx):= G$.
\end{lemma}
\begin{proof}
See Theorem 3.7.1 in \cite{hille1957functional} or Proposition 8.1
in \cite{hida1993white}.
\end{proof}

Let $T\subset\R$ be a time interval. In light of the above result,
we say an $(\SS)^*$-valued process $(F(t))_{t\in T}$ is in $ L^1(T)$ if
 \be\label{def:integrability}
 \lan\lan F(t),\psi\ran\ran\in L^1(T)\quad \text{for all }\psi\in(\SS).
 \ee

The Hida distribution space $(\SS)^*$ is a convenient space on which
to define the Wick product $F\diamond G\in (\SS)^*$.

\begin{definition}[Wick product]\label{def:wick product}
If  $F=\sum_{\beta\in\Lambda} b_\beta\HH_\bb$ and
$G=\sum_{\gamma\in\Lambda} c_\gamma\HH_\bb$ are elements of
$(\SS)^*$, then their {\it Wick product} is defined as
 \bea\nn
F\diamond G(\omega) &:=& \sum_{\beta,\gamma} b_{\beta} c_{\gamma}
\HH_{\beta+\gamma}(\omega)\\
&=&\nn\sum_{\lambda\in\Lambda}\(\sum_{\beta+\gamma=\lambda}
b_{\beta} c_{\gamma}\) \HH_{\lambda}(\omega), \ \ \text{for all }\omega\in \SS'(\R) .
 \eea
By Lemma 2.4.4 in \cite{holden1996}, $(\SS)^*$ is closed under this product.
\end{definition}

Given the measure $\PP$ on $\SS'$ from \eqref{eqn:Bochner}, we may
write Brownian motion as
 \bea
B(t)&=&\lan\omega,1_{[0,t]}(\cdot)\ran=\left\lan\omega,\sum_{k=1}^\infty\lan
1_{[0,t]},\xi_k\ran_{L^2(\R)}\xi_k(\cdot)\right\ran\\
&=&\nn \sum_{k=1}^\infty\int_0^t\xi_k(s)\, ds\, \lan\omega,\xi_k\ran
= \sum_{k=1}^\infty\int_0^t\xi_k(s)\, ds\, \HH_{\eps^{(k)}}(\omega).
 \eea
where $\eps^{(k)}$ is the multi-index with a $1$ in the $k$th entry
and $0$'s elsewhere. This motivates the following definition:

\begin{definition}[White noise]\label{def:white noise}
The $(\SS)^*$-valued process
$$W(t):= \sum_{k\ge 1}
\xi_k(t)\HH_{\eps^{(k)}}(\omega)$$ is called {\it white noise}.
\end{definition}

Using the Wick product on $(\SS)^*$ and the definition of the
integral of a Hida distribution given in Lemma
\ref{lem:hidaintegral}, we can now integrate a Hida distribution
with respect to white noise:
 \be
\int_\R F(t)\, dB_t := \int_\R F(t)\diamond W(t) \,dt.
 \ee
The above is called a WIS integral with respect to white noise. The following theorem shows that the WIS integral is a
generalization of the Skorohod integral:
\begin{theorem}\label{thm:skorohod reln}
Suppose $F(t,\omega):\R\times\Omega\to\R$ is Skorohod integrable.
Then \mbox{$F(t,\cdot)\diamond W(t)$} is $dt$-integrable in $(\SS)^*$ and
\be \int_a^b F(t,\omega) \,\delta B(t) = \int_a^b
F(t,\cdot)\diamond W(t)\,dt. \ee
\end{theorem}

\begin{proof}
See Theorem 2.5.9 in \cite{holden1996}.
\end{proof}

\subsection{Fractional white noise theory and a Fubini theorem}\label{app:fwn}

Elliot and Van Der Hoek \cite{elliot2003} introduced the fractional
white noise as an element of the Hida distribution space, and thus
constructed the WIS integral\footnotemark \footnotetext{Compare this to the integral
of \cite{hu2003fractional} where a fractional Brownian measure is
defined directly on tempered distributions $\SS'$. Integrals
with respect to this measure are defined for $H>1/2$.} which is valid for any
$H\in(0,1)$. The main tool
used to define the fractional white noise is the following operator\footnote{Our $M_H$ operator follows \cite{biagini2004} and differs from that of \cite{elliot2003} by a constant factor.} for which we set:

$$c_H:=\sqrt{\sin(\pi H) \Gamma(2H+1)}$$

\begin{definition}[$M_H$ operator]
The $M_H$ operator on $f\in\SS(\R)$ is defined by
$$\widehat{M_Hf}(y)= c_H|y|^{\frac{1}{2}-H}\hat{f}(y), \quad\text{for all } y\in\R.$$
\end{definition}
This operator extends to the space
 \bea\label{l2h}
L^2_H(\R)&:=&\{f:M_Hf\in L^2(\R)\}\\
&=&\nn\{f:|y|^{\frac{1}{2}-H}\hat{f}(y)\in L^2(\R)\}
 \eea
which is equipped with the inner product
 \be
 \lan f,g\ran_{L^2_H(\R)}=\lan M_Hf,M_Hg\ran_{L^2(\R)}.
 \ee
We note that $L^2_H(\R)$ is not closed under this inner product, and that its closure  contains distributions (see \cite[pg. 280]{nualart1995}
or \cite[Ch. 2]{biagini2008}).

Since $M_Hf\in L^2(\R)$ we can moreover define
$M_H:\SS'(\R)\to\SS'(\R)$ by
 \be
\lan M_H\omega,f\ran:=\lan \omega,M_Hf\ran\quad\text{for
}f\in\SS(\R), \omega\in\SS'(\R)
 \ee
 where $\lan \omega,M_Hf\ran$ is defined as in \eqref{S_to_L2}.

We have defined $M_H$ for a given fixed $H\in(0,1)$, however, the
theory extends to $M$ operators which are linear combinations
$a_1M_{H_1}+\cdots+a_nM_{H_n}$; for more on the $M$ operator see
\cite{samko1987integrals, elliot2003, lebovits2011stochastic}.

Using the orthonormal basis $e_k:=M_H^{-1}\xi_k$ of $L^2_H(\R)$ and writing $\rho^H_{0,t}:=M_H1_{[0,t]}$, we
define (and choose a continuous version) of fBm as
 \bea\label{def:fbmdef}
B^H(t)&:=&\lan\omega,\rho^H_{0,t}(\cdot)\ran = \lan
M_H\omega,1_{[0,t]}(\cdot)\ran
\eea
which satisfies, by \eqref{isometry} and (A.10) in \cite{elliot2003},
\bea
\E[B^H(t)B^H(s)]
&=&\lan \rho^H_{0,t}, \rho^H_{0,s}\ran_{L^2(\R)}\\
&=&\nn\frac{1}{2}(t^{2H} + s^{2H} - |t-s|^{2H}).
 \eea
 Note that the action of $\SS'$ on $\SS$ given by $\lan\omega,\cdot\ran$ is still inherited from the inner product of $L^2(\R)$ and not from the inner product of $L^2_H(\R)$.

The definition in \eqref{def:fbmdef} can be further rewritten as
 \bea
B^H(t)&=&= \lan
M_H\omega,1_{[0,t]}(\cdot)\ran\\
&=&\nn\left\lan M_H\omega,\sum_{k=1}^\infty\lan
1_{[0,t]},e_k\ran_{L^2_H(\R)}e_k(\cdot)\right\ran\\
&=&\nn\left\lan M_H\omega,\sum_{k=1}^\infty\lan
\rho^H_{0,t},\xi_k\ran_{L^2(\R)}e_k(\cdot)\right\ran\\
&=&\nn \sum_{k=1}^\infty\int_0^tM_H\xi_k(s)\, ds\,
\lan M_H\omega,e_k\ran = \sum_{k=1}^\infty\int_0^tM_H\xi_k(s)\, ds\,
\HH_{\eps^{(k)}}(\omega)
 \eea
which motivates the following notion of {\it fractional white
noise}:
 \be
W^H(t):= \sum_{k\ge 1} M_H\xi_k(t)\HH_{\eps^{(k)}}(\omega).
 \ee
By Lemma 4.1 in \cite{elliot2003}, $(W^H(t))_{t\in\R}$ is an $(\SS)^*$-valued process.
We note here that its underlying probability measure $\PP$ on $\SS'$ is the same as for
$(W(t))_{t\in\R}$.

\begin{definition}[WIS integral]\label{def:WIS integral}
Let $(F(t))_{t\in\R}$ be an $(\SS)^*$-valued process such that \mbox{$F(t)\diamond
W^H(t)$} is in $L^1(\R)$ (as in \eqref{def:integrability}). We define
 $$\int_\R F(t) \, dB^H_t :=\int_\R F(t) \diamond dB^H_t  := \int_\R F(t)\diamond W^H(t) \,dt.$$
\end{definition}

\begin{theorem}[fractional It\^o formula]\label{lemma:Ito}
Let $f:\R\times\R\to \R$, with $(s,x)\mapsto f(s,x)$, be in $C^{1,2}(\R\times\R)$. If the
random variables $$f(t,B^H_t),\ \ \int_r^t f_s(s,B^H_s) \,ds,\
\text{ and }\ \int_r^t f_{xx}(s,B^H_s) s^{2H-1}\, ds$$ are in
$L^2(\PP)$ for all $t>r$, then
\bea &&f(t,B^H_t) - f(r,B^H_r) \\&=& \int_r^t f_s(s,B^H_s)  \, ds +
\int_r^t f_x(s,B^H_s)  \, dB^H_s + H\int_r^t f_{xx}(s,B^H_s)  \, (s-r)^{2H-1}\,ds.\nn
\eea
\end{theorem}

\begin{proof}
See Theorem 3.8 of \cite{biagini2004}.
\end{proof}

We now prove a result which one can compare to Fubini-type theorems in
 \cite[Thm 3.7]{cheridito2005} and \cite[Thm 1.13.1]{mishura2008}.
Our result extends these Fubini theorems to integrals
of Hida distributions.

\begin{theorem}[Fubini-Tonelli theorem]\label{lemma:Fubini}
Let $$F_{s,r}=\sum_{\beta\in\Lambda} c_{\beta} (s,r) \HH_{\beta}$$ be an $(\SS)^*$-valued process indexed by
$(s,r)\in\R\times[0,t]$. If, for each $(\beta,k)$ pair,
$c_{\beta} (s,r)M_H\xi_k(s)$ is bounded above or below by an $L^1([r,t]\times[0,t])$ function, then
 \be\label{fubeqn}\int_0^t \int_r^t F_{s,r}(\omega)\,dB^{H}_s\,dr=
\int_0^t \(\int_0^s F_{s,r}(\omega)\,dr\) \,dB^{H}_s.
 \ee
 The equality in \eqref{fubeqn} is in the sense that if one side is in $(\SS)^*$, then the other is as well, and they are equal.  If in addition,
\be\label{eqn:intcondition}
 \(F_{s,r}1_{[r,t]}(s)\ddd W^{H}(s)\)(s,r)\in L^1(\R\times[0,t]),
\ee
 then both sides are in $(\SS)^*$ and  $c_\beta (s,\cdot)\in L^1[0,s]$ for a.e.
$s\in[0,t]$.
\end{theorem}
\begin{proof}
Unraveling the above definitions we have
\bea\label{eqn:fubini_calculations}
&&\int_0^t \int_r^t F_{s,r}(\omega)\,dB^{H}_s\,dr \\
\nn &:=& \int_0^t\int_\R F_{s,r}1_{[r,t]}(s) \diamond W^{H}(s)\,ds\,dr\\
 &:=& \int_0^t\int_\R
\nn\left(\sum_{\beta\in\Lambda} c_{\beta} (s,r) 1_{[r,t]}(s) \HH_{\beta}(\omega)\right)\diamond\left(\sum_{k\ge 0} M_H\xi_k(s)\HH_{\eps^{(k)}}(\omega)\right) \,ds\,dr\\
 &:=& \nn\int_0^t\int_\R \left(\sum_{\beta\in\Lambda, k\in\N} c_{\beta}(s,r)1_{[r,t]}(s) M_H\xi_k(s) \HH_{\beta+\eps^{(k)}}(\omega)\right) \,ds\,dr.
\eea
Denote the right-hand side above as $G$, an $(\SS)^*$-valued integral. By Lemma \ref{lem:hidaintegral},
such integrals are characterized by their action on $(\SS)$:
\bea\label{fubcalc}
&&\nn \lan\lan G,\psi\ran\ran \\&=&\nn \int_0^t \int_{\R} \sum_{\beta\in\Lambda, k\in\N}\lan\lan c_{\beta}(s,r)1_{[r,t]}(s)
M_H\xi_k(s)\HH_{\beta+\eps^{(k)}}, \psi \ran\ran \,ds\,dr\\
&=&\nn \sum_{\beta\in\Lambda, k\in\N}\int_0^t \int_{\R} \lan\lan c_{\beta}(s,r)1_{[r,t]}(s)
M_H\xi_k(s)\HH_{\beta+\eps^{(k)}}, \psi \ran\ran \,ds\,dr\\
 &=&\sum_{\beta\in\Lambda, k\in\N} \nn\int_\R\int_0^s  \lan\lan c_{\beta}(s,r)1_{[0,t]}(s) M_H\xi_k(s)\HH_{\beta+\eps^{(k)}}, \psi
 \ran\ran \,dr\,ds\\
  &=& \int_\R \sum_{\beta\in\Lambda, k\in\N} \lan\lan \lk\int_0^s c_{\beta}(s,r)1_{[0,t]}(s) M_H\xi_k(s)\HH_{\beta+\eps^{(k)}}  dr\rk, \psi
 \ran\ran\,ds
 \eea
where  equality in the second line follows since for any $\beta'\in\Lambda$ there are
finitely many pairs $(\beta,k)$ such that $\beta+\eps^{(k)}=\beta'$ (recall also that the $\HH_{\beta'}$ are orthogonal).
The third equality follows from Tonelli's Theorem for
real-valued functions which is possible due to our hypothesis. If in addition \eqref{eqn:intcondition} holds then both sides of this equality are in $(\SS)^*$ by \eqref{def:integrability}.
The final equality follows from Lemma
\ref{lem:hidaintegral}.

Teasing apart the right side of  \eqref{fubcalc} gives
 \bea
 \nn G&=& \int_\R\left(\sum_{\beta\in\Lambda} \lk\int_0^s c_{\beta} (s,r) dr\rk 1_{[0,t]}(s)\HH_{\beta}(\omega)\right)\diamond\left(\sum_{k\ge 0}
M_H\xi_k(s)\HH_{\eps^{(k)}}(\omega)\right)\,ds \\
 &=&\int_0^t \(\int_0^s F_{s,r}(\omega)\,dr\) \,dB^{H}_s
\eea
 as needed.
\end{proof}

\section{Second moment bounds for fBm}\label{appendix}

Our aim in this section is to prove Lemma \ref{lab}, but we must first take care of a technical detail. Recall
\bea \label{tor2}
\lambda &:=& \mbox{Var}(B^H_{s}-B^H_r) = |s-r|^{2H}
\\ \nn  \ppp &:=& \mbox{Var}(B^H_{s'}-B^H_{r'}) = |s'-r'|^{2H}
\\ \nn  \mu &:=& \mbox{Cov}(B^H_{s}-B^H_r,B^H_{s'}-B^H_{r'}) \\
&=&\nn \frac{1}{2}(|s'-r|^{2H} + |s-r'|^{2H} - |s'-s|^{2H} - |r'-r|^{2H}).
\eea

As before, we let $K$ denote a positive constant,
which may change from line to line. Lemma 3.1 of \cite{hu2001} asserts the following bounds on the quantity $\la \ppp - \mu^2$, which have been found useful by a number of authors in connection with intersection local times of fBm.

{\it
\begin{itemize} \label{}

\item[(i)] Suppose $r<r'<s<s'$. Let $a=r'-r$, $b=s-r'$, $c=s'-s$. Then

\be \lll{o1}
\la \ppp - \mu^2 \geq K\((a+b)^{2H}c^{2H} + a^{2H}(b+c)^{2H}\).
\ee

\item[(ii)] Suppose $r<r'<s'<s$. Let $a=r'-r$, $b=s'-r'$, $c=s-s'$. Then

\be \lll{oldone}
\la \ppp - \mu^2 \geq Kb^{2H}(a+b+c)^{2H}.
\ee

\item[(iii)] Suppose $r<s<r'<s'$. Let $a=s-r$, $b=r'-s$, $c=s'-r'$. Then

\be \lll{o3}
\la \ppp - \mu^2 \geq K(a^{2H}c^{2H}).
\ee

\end{itemize}
}

\vski

Unfortunately, $(ii)$ is false, which can be seen by noting that if
$r' \searrow r, s' \nearrow s$ then the left side of \rrr{oldone}
approaches 0 while the right side does not. Fortunately, however, $(ii)$ can be
replaced by $(ii')$ in the following lemma, which suffices for our purposes as well as in every other instance in
which we have seen $(ii)$ applied.

\vski

\bl \label{lemma:appendix} (i) and (iii) hold, and (ii) may be replaced by the following:

\vski

(ii') Suppose $r<r'<s'<s$. Let $a=r'-r$, $b=s'-r'$, $c=s-s'$. Then

\be \lll{newone}
\la \ppp - \mu^2 \geq Kb^{2H}(a^{2H}+c^{2H}).
\ee

\el

\begin{proof}[Proof of (ii')]

We will follow the method of proof given in \cite{hu2001} for $(ii)$ to arrive at
$(ii')$. We use the local nondeterminism property of fBm \cite{xiao2006} which implies that the following is true
for $t_0< \ldots < t_j$:

\be \label{jc}
\mbox{Var}\(\sum_{i=1}^j u_i (B^H_{t_i}-B^H_{t_{i-1}}) \) \geq K
\sum_{i=1}^j |u_i|^2 (t_i-t_{i-1})^{2H}.
\ee
In our case, this implies

\bea \label{k0}
&& \mbox{Var}\( u (B^H_{s}-B^H_{r})+v(B^H_{s'}-B^H_{r'})\)
\\ \nn && \hspace{1cm} = \mbox{Var}\( u (B^H_{r'}-B^H_{r})+
(u+v)(B^H_{s'}-B^H_{r'}) + u (B^H_{s}-B^H_{s'})\)
\\ \nn && \hspace{1cm} \geq K\(u^2(a^{2H} +c^{2H}) + (u+v)^2(b^{2H})\).
\eea
However, we also have

\be \lll{k1}
\mbox{Var}\( u (B^H_{s}-B^H_{r})+v(B^H_{s'}-B^H_{r'})\) = u^2\la + 2uv\mu + v^2 \ppp.
\ee
Combining \rrr{k0} and \rrr{k1} yields

\be \lll{aic}
u^2(\la - K a^{2H} - Kb^{2H} - Kc^{2H}) + 2uv(\mu - Kb^{2H}) +v^2(\ppp - Kb^{2H}) \geq 0.
\ee
The discriminant of \rrr{aic} therefore satisfies

\be \lll{}\nn
4(\mu - Kb^{2H})^2 - 4(\la - K a^{2H} - Kb^{2H} - Kc^{2H})(\ppp - Kb^{2H}) \leq 0
\ee
which gives
\bea \lll{d9}
&& \la \ppp - \mu^2
\\ \nn \hspace{1cm} &\geq& K\((a^{2H} + b^{2H} + c^{2H})\ppp + b^{2H}\la - 2\mu b^{2H}\)
\\ \nn && \hspace{1.2cm} -K^2\(b^{2H}(a^{2H} + b^{2H} + c^{2H})-b^{4H}\).
\eea
By the Cauchy-Schwarz inequality,

\be \lll{}\nn
\mu \leq \sqrt{\la \ppp} \leq \frac{\la + \ppp}{2}.
\ee
Using this, together with $\la = (a+b+c)^{2H}$ and $\ppp = b^{2H}$ allows us to reduce \rrr{d9} to

\be\nn
\la \ppp - \mu^2 \geq (K-K^2)b^{2H}(a^{2H} + c^{2H}).
\ee
The result follows by replacing $(K-K^2)$ with $K$.
\end{proof}

{\it Proof of Lemma \ref{lab}:} Recall that we must show for $\DD_t=\{0\leq
r\leq s \leq t\}$ and $0<H<2/3$ we have

\be \lll{t43}
\int_{\DD_t^2}\frac{\mu (s-r)^{2H-1}(s'-r')^{2H-1}}{(\la \ppp - \mu^2)^{3/2}}  \,dr\,dr'\,ds\,ds'<\ff.
\ee
We will split the range of integration into the three regions described in in Lemma \ref{lemma:appendix}, with $a,b,c$ defined accordingly on each region.

\vski

{\bf Case 1:} $r<r'<s<s'$. Using $s-r=a+b, s'-r'=b+c$, and \rrr{o1} we see that the contribution of this region to \rrr{t43} can be bounded by

\begin{equation} \label{noga}
\begin{split}
K\int_{[0,t]^3} &\frac{\mu(a+b)^{2H-1}(b+c)^{2H-1}}{((a+b)^{2H}c^{2H} + a^{2H}(b+c)^{2H})^{3/2}}\,da\,db\,dc \\
& \leq K\int_{[0,t]^3} \frac{\mu(a+b)^{2H-1}(b+c)^{2H-1}}{(a+b)^{3H/2}c^{3H/2}a^{3H/2}(b+c)^{3H/2}}\,da\,db\,dc \\
& = K\int_{[0,t]^3} \frac{\mu \,da\,db\,dc}{(a+b)^{1-H/2}c^{3H/2}a^{3H/2}(b+c)^{1-H/2}}.
\end{split}
\end{equation}
Using an idea which appears in \cite{hu2001}, we write

\begin{equation} \label{oasis}
\begin{split}
2\mu & = (a+b+c)^{2H}+b^{2H} - a^{2H} - c^{2H} \\
& = 2H(b+c) \int_{0}^{1} (a+(b+c)u)^{2H-1}du + b^{2H} - c^{2H}.
\end{split}
\end{equation}
We can bound the integral in \rrr{oasis} by replacing the integrand by its maximal value.
This gives a bound of $|\mu| \leq K((b+c)a^{2H-1} +b^{2H} + c^{2H})$ for $H < 1/2$, and a bound of $|\mu| \leq K((b+c)+b^{2H} + c^{2H})$ when $H \geq 1/2$.

If $H<1/2$, we bound \rrr{noga} by

\begin{equation} \label{}
K\int_{[0,t]^3} \frac{((b+c)a^{2H-1} +b^{2H} + c^{2H})}{(a+b)^{1-H/2}c^{3H/2}a^{3H/2}(b+c)^{1-H/2}}\,da\,db\,dc.
\end{equation}
We need only show that each of these three terms is integrable, and we may obtain bounds by replacing nonnegative powers of
$(a+b)$ in the denominator by powers of either $a$ or $b$, and likewise for $(b+c)$.
When $H<1/2$ we will also use $(a+b)^{1-H/2} \geq a^{1-2H}b^{3H/2}$.
Integrability follows from:

\begin{eqnarray} \label{a15}\nn
\frac{(b+c)a^{2H-1}}{(a+b)^{1-H/2}c^{3H/2}a^{3H/2}(b+c)^{1-H/2}} &\leq& \frac{K}{(a+b)^{1-H/2}c^{3H/2}a^{1-H/2}}\leq \frac{K}{b^{1-H/2}c^{3H/2}a^{1-H/2}},\\ \nn
\frac{b^{2H}}{(a+b)^{1-H/2}c^{3H/2}a^{3H/2}(b+c)^{1-H/2}} &\leq& \frac{K b^{2H}}{b^{3H}c^{1-H/2}a^{1-H/2}}= \frac{K}{b^{H}c^{1-H/2}a^{1-H/2}},\\
\frac{c^{2H}}{(a+b)^{1-H/2}c^{3H/2}a^{3H/2}(b+c)^{1-H/2}} &\leq& \frac{K }{(a+b)^{1-H/2}c^{1-H}a^{3H/2}}\leq \frac{K }{b^{1-H/2}c^{1-H}a^{3H/2}}.\nn\\
\end{eqnarray}
In each case we obtain an integrable function of $a,b,c$, i.e. the exponent of each variable in the denominator is less than one.

For $1/2 \leq H < 2/3$, we have

\begin{equation*} \label{}
\frac{(b+c)}{(a+b)^{1-H/2}c^{3H/2}a^{3H/2}(b+c)^{1-H/2}} \leq \frac{K}{(a+b)^{1-H/2}c^{3H/2}a^{3H/2}}\leq \frac{K}{b^{1-H/2}c^{3H/2}a^{3H/2}},
\end{equation*}

\begin{equation*} \label{}
\frac{b^{2H}}{(a+b)^{1-H/2}c^{3H/2}a^{3H/2}(b+c)^{1-H/2}} \leq \frac{K }{b^{2-3H}c^{3H/2}a^{3H/2}},
\end{equation*}
while the third term follows exactly as in the case $H<1/2$ (note the second inequality in \eqref{a15}
used a bound which is not valid when $H>1/2$). This completes Case 1.
\vski

{\bf Case 2:} $r<r'<s'<s$. Following \cite{hu2001} we can write

\begin{equation} \label{}
\begin{split}
2\mu & = (a+b)^{2H}-a^{2H} + (b+c)^{2H} - c^{2H} \\
& = 2Hb \int_{0}^{1} ((a+bu)^{2H-1} + (c+bu)^{2H-1})\,du.
\end{split}
\end{equation}
Replacing the integrand with its maximum over the region gives us a bound of $\mu \leq K b$ for $H \geq 1/2$ and $\mu \leq K b (a^{2H-1}+c^{2H-1})$ for $H<1/2$.

First consider $H< 1/2$. In this case, $s-r=a+b+c$ and  $s'-r'=b$.  Also note that
$$(a+c)^{2H} \le (2\max(a,c))^{2H} =2^{2H} \max(a,c)^{2H} \le 2^{2H} (a^{2H} + c^{2H}),$$
thus  $(a+c)^{2H}$ and $(a^{2H} + c^{2H})$ are equivalent up to a constant.
We bound the contribution of this region to \rrr{t43}, using \rrr{newone}, by

\begin{equation} \label{}
\begin{split}
K\int_{[0,t]^3} & \frac{b (a^{2H-1}+c^{2H-1})(a+b+c)^{2H-1}b^{2H-1}}{b^{3H}(a+c)^{3H}}\,da\,db\,dc \\
& \leq K \int_{[0,t]^3} \frac{a^{2H-1}(a+b+c)^{2H-1}}{b^{H}(a+c)^{3H}}\,da\,db\,dc \\
& \leq K \int_{[0,t]^3} \frac{a^{2H-1}b^{2H-1}}{b^{H}(a+c)^{3H}}\,da\,db\,dc \\
& \leq K \int_{[0,t]^2} \frac{a^{2H-1}}{(a+c)^{3H}}\,da\,dc \\
& \leq K \int_{[0,t]} a^{2H-1}(1+\frac{1}{a^{3H-1}})\,da < \ff.
\end{split}
\end{equation}
For $1/2\leq H<2/3$, we have

\begin{equation} \label{}
\begin{split}
K\int_{[0,t]^3} & \frac{b(a+b+c)^{2H-1}b^{2H-1}}{b^{3H}(a+c)^{3H}}\,da\,db\,dc \\
& = K \int_{[0,t]^3} \frac{(a+b+c)^{2H-1}}{b^{H}(a+c)^{3H}}\,da\,db\,dc \\
& \leq K \int_{[0,t]^3} \frac{\,da\,db\,dc}{b^{H}(a+c)^{3H}} < \ff.
\end{split}
\end{equation}
This completes Case 2.
\vski

{\bf Case 3:} $r<s<r'<s'$. Following \cite{hu2001} we write

\begin{equation} \label{rem}
\begin{split}
2\mu & = (a+b+c)^{2H}+b^{2H} - (a+b)^{2H} - (b+c)^{2H} \\
& = 2H(2H-1)ac \int_{0}^{1}\int_{0}^{1} (b+vc+ua)^{2H-2}\,du\,dv.
\end{split}
\end{equation}
By Young's inequality, applied twice, for $\aaa,\bb>0$ with $\aaa + \bb = 1$ we have

\begin{equation} \label{}
\begin{split}
(b+vc+ua) \geq Kb^\aa(vc+ua)^\bb \geq K b^\aa(vc)^{\bb/2}(ua)^{\bb/2}.
\end{split}
\end{equation}
This combines with \rrr{rem} to give

\begin{equation} \label{live}
|\mu| \leq K (ac)^{\bb(H-1) + 1}b^{2\aa(H-1)}.
\end{equation}
Using $s-r=a, s'-r'=c$, \rrr{live}, and \rrr{o3} shows that the contribution of this region to \rrr{t43} is bounded by

\begin{equation} \label{davmat}
\begin{split}
K\int_{[0,t]^3} & \frac{(ac)^{\bb(H-1) + 1}b^{2\aa(H-1)} a^{2H-1}c^{2H-1}}{a^{3H}c^{3H}}\,da\,db\,dc \\
& = K\int_{[0,t]^3} \frac{\,da\,db\,dc}{b^{2\aa(1-H)}(ac)^{\bb + H(1-\bb)}}.
\end{split}
\end{equation}
Note that $1>2H(1-H)$ so that $\frac{1}{2(1-H)}>H$. We may therefore choose $\aa$ such that $H< \aa < \frac{1}{2(1-H)}$ implying $\bb < 1-H$.
The exponents in the final expression in \rrr{davmat} can therefore be bounded by $2\aa(1-H)< 2(\frac{1}{2(1-H)})(1-H)=1$
and $\bb + H(1-\bb) < (1-H) + H = 1$. We conclude that \rrr{davmat} is finite, which completes the final case for the proof of Lemma \ref{lab}. \qed
\noindent

\bibliographystyle{alpha}
\bibliography{biblioJan2013}
\end{document}